\documentclass{amsart}

\usepackage{amsmath}%
\usepackage{amsfonts}%
\usepackage{amssymb}%
\usepackage{graphicx}
\usepackage[latin1]{inputenc}
\usepackage[T1]{fontenc}
\newtheorem{theorem}{Theorem}
\theoremstyle{plain}

\newtheorem{corollary}{Corollary}

\newtheorem{definition}{Definition}
\newtheorem{example}{Example}

\newtheorem{proposition}{Proposition}
\newtheorem{remark}{Remark}

\numberwithin{equation}{section}
\begin{document}
\title[ $\epsilon$-isothermic surfaces in pseudo-Euclidean 3-space]{$\epsilon$-isothermic surfaces in pseudo-Euclidean 3-space}

\author{Armando M. V. Corro}
\curraddr[Armando M. V. Corro]{Instituto de Matem\'atica e Estat\'istica, Universidade Federal de Goi\^as, 74001-970, Goi\^ania-GO, Brazil.}%
\email[]{avcorro@gmail.com}%
\author{Carlos M. C. Riveros}
\curraddr[Carlos M. C. Riveros]{Departamento de Matem\'atica, Universidade de Bras\'ilia, CEP 70910-900, Bras\'ilia, DF, Brasil.}%
\email[]{carlos@mat.unb.br}%
\author{Marcelo L. Ferro}
\curraddr[Marcelo L. Ferro]{Instituto de Matem\'atica e Estat\'istica, Universidade Federal de Goi\^as, 74001-970, Goi\^ania-GO, Brazil.}%
\email[]{marceloferro@ufg.br}%
%\thanks{Thanks for Author One.}
%\thanks{Thanks for Author Two.}
%\date{March 15, 2002}
\subjclass[2010]{53A35} %
\keywords{Dupin surfaces, isothermic surfaces, lines of curvature.}%
%\dedicatory{Dedicated to Professor XY on the occasion of his seventieth birthday.}

\begin{abstract}
In this paper we describe the $\epsilon$-isothermic surfaces in the pseudo-Euclidean 3-space and we obtain the pseudo-Calapso equation. In sequence, we classify the Dupin surfaces in pseudo-Euclidean 3-space having distinct principal curvatures and provide explicit coordinates for such surfaces. As application of the theory, we give explicit solutions to the pseudo-Calapso equation.
\end{abstract}
\maketitle
\section{Introduction}
Dupin surfaces were first studied by Dupin in 1822. A Hypersurface is said to be Dupin  if each principal curvature is constant along
its corresponding surface of curvature. A Dupin submanifold $M$ is said to be proper if the number $g$ of distinct principal
curvatures is constant on $M$. The simplest Dupin submanifolds are the isoparametric hypersurfaces, that is, those whose principal curvatures are constant.

Dupin's surfaces in Euclidean space are classified. There are several equivalent definitions of Dupin cyclides, for example, in Euclidean space, they can be defined as any inversion of a torus, cylinder or double cone, i.e, Dupin cyclide is invariant under M\"obius transformations. Classically the cyclides of Dupin were characterized by the property that both sheets of the focal set are curves. Another equivalent definition says that such surfaces can also be given as surfaces that are the envelope of two families at 1-parameter spheres (including planes as degenerate spheres). For more on Dupin cyclides see \cite{1}-\cite{2}.

The research of isothermic surfaces is one of the most common
more difficult problems of differential geometry and depends on the integration of an equation with fourth-order partial derivatives (see \cite{We}).
Particular classes of these surfaces are known and some transformations by means of which it is possible to deduce from isothermic surfaces other isothermic surfaces. All this is known indirectly and independently of the fourth-order differential equation, because it is difficult to integrate.

The theory of isothermic surfaces has a great development for eminent geometers as Christoffel \cite{Cr}, Darboux \cite{D1}, \cite{D2} and Bianchi \cite{Bi} among others.
In the last decades, the theory woke up interest by his connection with the modern theory of integrated systems, see \cite{Ci1}, \cite{Ci2}, \cite{Ro}, \cite{So} and \cite{Sy}.
Particular classes of isothermic surfaces are the constant mean curvature surfaces, quadrics, surfaces whose lines of curvature has constant geodesic curvature, in particular, the cyclides of Dupin. Trasformations of $\mathbb{R}^{3}$ that preserve isothermic surfaces are isometries, dilations and inversions.

In \cite{Bo}, the authors study surfaces with harmonic inverse mean curvature (HIMC surfaces), they distinguish a subclass of $\theta$-isothermic surfaces,
which is a generalization of the isothermic HIMC surfaces, and classify all the $\theta$-isothermic HIMC
surfaces, note that when $\theta=0$, the surfaces are isothermic.

In \cite{Ci2}, the author show that theory of soliton surfaces, modified in an appropriate way, can be applied also to isothermic immersions in $\mathbb{R}^3$. In this case the so called Sym's formula gives an explicit expression for the isothermic immersion with prescribed fundamental forms.
The complete classification of the isothermic surfaces is an open problem.

In \cite{3} the author establishes an equation with fourth order partial derivatives from which the problem of obtaining isothermic surfaces apparently becomes much simpler. Such equation ( called Calapso equation ) defined in \cite{3}  given by
\begin{eqnarray}\nonumber
\Delta\bigg(\frac{\phi,_{12}}{\phi}\bigg)+\big(\phi^2\big),_{12}=0,\nonumber
\end{eqnarray}
 describes isothermic surfaces in $\mathbb{R}^3$, where $\phi,_{12}$ denotes the derivative of $\phi$ with respect to $u_1$ and $u_2$.

In \cite{13} the authors introduced the class of radial inverse mean curvature surface (RIMC-surfaces ), that are isothermic surfaces. Moreover, were obtained two solutions of the Calapso equation where one can be obtained using \cite{1} and a different one.

In this paper, motivated by \cite{3} and \cite{13} we describe the $\epsilon$-isothermic surfaces in the pseudo-Euclidean 3-space and we obtain the pseudo-Calapso equation
\begin{eqnarray}\nonumber
\Delta_\epsilon\bigg(\frac{\phi,_{12}}{\phi}\bigg)+\epsilon_2\big(\phi^2\big),_{12}=0,\nonumber
\end{eqnarray}
 where $\epsilon_1^2=\epsilon_2^2=1$, $\epsilon=\epsilon_1\epsilon_2$ and $\phi,_{12}$ denotes the derivative of $\phi$ with respect to $u_1$ and $u_2$.
For each $\epsilon$-isothermic surface of the pseudo-Euclidean 3-space, we show that we can associate to these surfaces two solutions for the pseudo-Calapso equation. In sequence, we consider those proper Dupin surface of the pseudo-Euclidean 3-space having distinct principal curvatures, parametrized by lines of curvature. We prove that every Dupin surface parametrized by lines of curvature has a $\epsilon$-isothermic surface and provide explicit coordinates for such surfaces. As application of the theory, we give explicit solutions of the pseudo-Calapso equation.

\section{$\epsilon$-isothermic surfaces and the pseudo-Calapso equation}

In this section, we briefly review the main definitions, we describe the $\epsilon$-isothermic surfaces in the pseudo-Euclidean 3-space and we obtain
the pseudo-Calapso equation.

We consider $E^3$ as the pseudo-Euclidean 3-space, i.e, $\mathbb{R}^{3}$ equipped with the metric $\langle ,\rangle$, given by
$$\langle(x_1,y_1,z_1),(x_2,y_2,z_2)\rangle=\epsilon_1x_1x_2+\epsilon_2y_1y_2+\epsilon_3z_1z_2$$
where $\epsilon_i^2=1$, $1\leq i\leq 3$.

In particular, making $\epsilon_1=1$, we consider $\mathbb{C}_{\epsilon_2}$ as the pseudo-complex space with the metric $z\overline{z}=u_1^2+\epsilon_2 u_2^2$, where $z=u_1+iu_2$, $i^2=-\epsilon_2$.
\begin{definition}{\em
We define the function $f:\mathbb{C}_{\epsilon_2}\rightarrow \mathbb{C}_{\epsilon_2}$, $f(z)=f(u_1+iu_2)=u(u_1,u_2)+iv(u_1,u_2)$ as {\em $\epsilon_2-$holomorphic} if and only if $u_{,1}=v_{,2}$ and $u_{,2}=-\epsilon_2v_{,2}$ (see \cite{17}).}
\end{definition}
Here the subscript $,i$ denotes the derivative with respect to $u_i$.
\begin{definition}{\em
 A surface $M$ in pseudo-Euclidean 3-space is called \textit{$\epsilon$-isothermic surface} if it admits parametrization by lines of
 curvature and the first fundamental form is conformal to the metric $\epsilon_1 du_1^2+ \epsilon_2 du_2^2$.}
\end{definition}
\begin{remark}{\em
 If $X:U\subseteq \mathbb{R}^2\rightarrow E^{3}$ is a parametrization of a $\epsilon$-isothermic surface $M$, then the first and the second
 fundamental forms are given by
\begin{eqnarray}\label{eq1forma}
I=e^{2\varphi}\big(\epsilon_1du_1^2+\epsilon_2du_2^2\big)\hspace{1cm}II=edu_1^2+gdu_2^2,
\end{eqnarray}
where $\epsilon_i^2=1$, $1\leq i\leq 2$.\\
The Codazzi and Gauss equation are given by
\begin{eqnarray}
e,_2&=&\big(e+\epsilon g\big)\varphi,_2\nonumber\\
g,_1&=&\epsilon\big(e+\epsilon
g\big)\varphi,_1\label{gausscodazi}\\
\Delta_\epsilon\varphi &=&-\epsilon_2ege^{-2\varphi}\label{gaussequation},
\end{eqnarray}
where
$\Delta_\epsilon\varphi=\varphi,_{11}+\epsilon\varphi,_{22}$, with
$\epsilon=\epsilon_1\epsilon_2$.\\
To integrate the system (\ref{gausscodazi}) we make the following
substitution
\begin{eqnarray}\label{eqint}
e=\frac{1}{\sqrt{2}}\big(\omega+\Omega\big)e^\varphi\hspace{1cm}g=\frac{\epsilon}{\sqrt{2}}\big(\omega-\Omega\big)e^\varphi.
\end{eqnarray}
Thus, by (\ref{eqint}) the system (\ref{gausscodazi}) can be
written as
\begin{eqnarray}
\Omega,_1&=&\omega,_1-\big(\omega+\Omega\big)\varphi,_1\label{eqint1}\\
\Omega,_2&=&-\omega,_2+\big(\omega-\Omega\big)\varphi,_2\label{eqint2}\\
\Delta_\epsilon\varphi
&=&\frac{\epsilon_2}{2}\big(\Omega^2-\omega^2\big),\label{eqint3}
\end{eqnarray}
where $\epsilon=\epsilon_1\epsilon_2$.}
\end{remark}
\begin{definition}{\em
 For each function $\omega=\omega(u_1, u_2)$, we define the {\em pseudo-Calapso equation} as
\begin{eqnarray}\label{calapsoequation}
\Delta_\epsilon\bigg(\frac{\omega,_{12}}{\omega}\bigg)+\epsilon_2\big(\omega^2\big),_{12}=0,
\end{eqnarray}
where $\epsilon=\epsilon_1\epsilon_2$ and $\epsilon_1^2=\epsilon_2^2=1$.}
\end{definition}
\begin{definition}{\em
Let $M$ be a surface with principal curvatures $-\lambda_1$ and $-\lambda_2$. The \textit{skeaw curvature} of $M$ (see \cite{13})
is given by
\begin{eqnarray}\label{skeaw}
H'=\frac{\lambda_2-\lambda_1}{2}.
\end{eqnarray}}
\end{definition}

\begin{theorem}{Let $X(u_1, u_2)$ be a $\epsilon$-isothermic surface with first
fundamental form given by
\begin{eqnarray*}\nonumber
I=e^{2\varphi}\big(\epsilon_1du_1^2+\epsilon_2du_2^2\big).\nonumber
\end{eqnarray*}
Then the functions
\begin{eqnarray}\label{calpsosolution}
\omega=\epsilon_1\sqrt{2}{e\varphi}H\hspace{0,5cm}and\hspace{0,5cm}\Omega=\epsilon_1\sqrt{2}{e\varphi}H',
\end{eqnarray}
are solutions of the pseudo-Calapso equation (\ref{calapsoequation}), where $\epsilon=\epsilon_1\epsilon_2$, $H$ is the mean curvature
and $H'$ the skew curvature of $X$.}
\end{theorem}

\begin{proof} Differentiating the equation (\ref{eqint1})
with respect to $u_2$ and the equation (\ref{eqint2}) with respect to
$u_1$, adding and subtracting these expression, we obtain
\begin{eqnarray}\label{eq1teo1.1}
\frac{\omega,_{12}}{\omega}=\varphi,_{12}+\varphi,_1\varphi,_2,\hspace{1cm}\frac{\Omega,_{12}}{\Omega}=-\varphi,_{12}+\varphi,_1\varphi,_2.
\end{eqnarray}
Let
\begin{eqnarray}\label{eq2teo1.2}
A=\Delta_\epsilon\bigg(\frac{\omega,_{12}}{\omega}\bigg)+\epsilon_2\big(\omega^2\big),_{12},
\end{eqnarray}
using the first equation of (\ref{eq1teo1.1}) and properties
of the Laplacian operator, we get
\begin{eqnarray*}\nonumber
A=\epsilon_2\big(\omega^2\big),_{12}+\big(\Delta_\epsilon\varphi\big),_{12}+\varphi,_2\big(\Delta_\epsilon\varphi\big),_{1}+
\varphi,_1\big(\Delta_\epsilon\varphi\big),_{2}+2\varphi,_{12}\big(\Delta_\epsilon\varphi\big),
\end{eqnarray*}
From (\ref{eqint3}), we have
\begin{eqnarray*}\nonumber
A=\frac{\epsilon_2}{2}\bigg[\big(\omega^2+\Omega^2\big),_{12}+\big(\Omega^2-\omega^2\big),_{1}\varphi,_2+\big(\Omega^2-\omega^2\big),_{2}\varphi,_1+
\big(\Omega^2-\omega^2\big)\varphi,_{12}\bigg].
\end{eqnarray*}
Substituting $(\omega^2+\Omega^2),_{12}$ and using
(\ref{eq1teo1.1}), we obtain
\begin{eqnarray*}\nonumber
A=\frac{\epsilon_2}{2}\bigg[2\big(\omega,_1\omega,_2+\Omega,_1\Omega,_2\big)+2\big(\Omega^2+\omega^2\big)\varphi,_2\varphi,_1+
\big(\Omega^2-\omega^2\big),_{1}\varphi,_2+\big(\Omega^2-\omega^2\big),_{2}\varphi,_1\bigg].
\end{eqnarray*}
Substituting $(\Omega^2-\omega^2),_{1}$,
$(\Omega^2-\omega^2),_{2}$ and using (\ref{eqint1}),
(\ref{eqint2}) and (\ref{eqint3}), we obtain $A=0$. Using
(\ref{eq2teo1.2}) and (\ref{eqint}), we get that
$\omega=\epsilon_1\sqrt{2}e^\varphi H$ is a solution to the Pseudo-Calapso equation.\\
On the other hand, using (\ref{eq1teo1.1}) we get
\begin{eqnarray*}\nonumber
\Delta_\epsilon\bigg(\frac{\omega,_{12}}{\omega}\bigg)-\Delta_\epsilon\bigg(\frac{\Omega,_{12}}{\Omega}\bigg)=2\Delta_\epsilon\big(\varphi,_{12}\big)=
2\big(\Delta_\epsilon\varphi\big),_{12}=\epsilon_2\big(\Omega^2-\omega^2\big),_{12}.
\end{eqnarray*}
This last equation is equivalent to
\begin{eqnarray*}\nonumber
\Delta_\epsilon\bigg(\frac{\omega,_{12}}{\omega}\bigg)+\epsilon_2\big(\omega^2\big),_{12}=\Delta_\epsilon\bigg(\frac{\Omega,_{12}}{\Omega}\bigg)+\epsilon_2\big(\Omega^2\big),_{12}.
\end{eqnarray*}
Since $ \omega$ is a solution for the pseudo-Calapso equation we obtain that $\Omega=\epsilon_1\sqrt{2}e^\varphi H'$ is also a solution of the pseudo-Calapso
equation. The proof is complete.

\end{proof}

\section{Classification of Dupin surfaces in Pseudo-Euclidean 3-space}

In this section, we provide a classification of Dupin surfaces parametrized by lines of curvature in pseudo-Euclidean 3-space $E^3$, with two distinct principal curvatures.
\begin{definition}{\em
An immersion $X:U\subseteq \mathbb{R}^{2}\rightarrow E^{3}$ is a \textit{Dupin surface} if each principal curvature is
constant along its corresponding line of curvature.}
\end{definition}

Let $X:U\subseteq \mathbb{R}^{2}\rightarrow E^{3}$, be a proper Dupin surface parametrized by lines of curvature, with distinct
principal curvature, $-\lambda_i$, $1\leq i\leq 2$, and let $N:U\subseteq \mathbb{R}^2\rightarrow E^{3}$ be a unit
normal vector field of X. Then\\
\begin{eqnarray}
\langle X,_i,X,_j\rangle &=&\delta_{ij}g_{ij},\hspace{1cm} 1\leq i, j\leq
2,
\label{metricageral}\\
N,_i&=&\lambda_iX,_i,\hspace{1cm} 1\leq i\leq
2,\label{Nigeral}\\
\langle N,N\rangle &=& \epsilon_3,\nonumber\\
 \lambda_{i,i}&=&0,\nonumber
\end{eqnarray}
where $\langle , \rangle$ denotes the pseudo-Euclidean 3-space on $E^{3}$,
$\epsilon_j^2=1$, $1\leq j\leq3$.

Moreover for $1\leq i\neq j\leq 2$, we have
\begin{eqnarray}
&&X,_{ij}-\Gamma^i_{ij}X,_i-\Gamma^j_{ij}X,_j=0, \label{eqXegeral}\\
&&\Gamma^i_{ij}=\frac{\lambda_{i,j}}{\lambda_j-\lambda_i},\label{codazi}
\label{simcristgeral}
\end{eqnarray}
where $\Gamma^i_{ij}$ are the Christoffel symbols.

The Christoffel symbols in terms of the metric
(\ref{metricageral}) are given by
\begin{equation}\label{simbmetrica}
\Gamma^i_{ii}=\frac{g_{ii,i}}{2g_{ii}},\hspace{0,4cm}\Gamma^j_{ii}=-\frac{g_{ii,j}}{2g_{jj}},\hspace{0,4cm}\Gamma^i_{ij}=\frac{g_{ii,j}}{2g_{ii}},
\end{equation}
where $1\leq i, j\leq 2$ are distinct.\\
 It follows from
(\ref{simbmetrica}), that
\begin{eqnarray}\label{gammaiij}
\Gamma_{ii}^j=-\Gamma_{ij}^i\frac{g_{ii}}{g_{jj}}
\end{eqnarray}
From (\ref{Nigeral}) and (\ref{gammaiij}), we get
\begin{equation}\label{xiigeral}
X,_{ii}=\Gamma_{ii}^iX,_i-\Gamma_{ij}^i\frac{g_{ii}}{g_{jj}}X,_j-\lambda_ig_{ii}N.
\end{equation}

\begin{theorem} {Let $X:U\subseteq \mathbb{R}^{2}\rightarrow E^{3}$, be a Dupin surface parametrized by lines
of curvature, with two distinct principal curvatures $-\lambda_1$ and $-\lambda_2$. Then there is a change in each coordinate
separately so that the first fundamental form is given by
\begin{eqnarray}\label{1forma}
I=\frac{1}{(\lambda_2-\lambda_1)^2}\bigg(\epsilon_1du_1^2+\epsilon_2du_2^2\bigg),
\end{eqnarray}
i.e. $X$ is a $\epsilon$-isothermic surface.}
\end{theorem}

\begin{proof} Using (\ref{codazi}) and (\ref{simbmetrica}), we have
\begin{eqnarray*}\nonumber
\frac{\lambda_{1,2}}{\lambda_2-\lambda_1}=\frac{g_{11,2}}{2g_{11}},\hspace{1cm}\frac{\lambda_{2,1}}{\lambda_1-\lambda_2}=\frac{g_{22,1}}{2g_{22}}.
\end{eqnarray*}
These last two equations can be rewritten as
\begin{eqnarray*}\nonumber
\bigg[\ln\bigg(\frac{1}{\lambda_2-\lambda_1}\bigg)^2\bigg]_{,2}=\big(\ln|g_{11}|\big)_{,2},\hspace{1cm}\bigg[\ln\bigg(\frac{1}{\lambda_2-\lambda_1}\bigg)^2\bigg]_{,1}=\big(\ln|g_{22}|\big)_{,1},
\end{eqnarray*}
so that
\begin{eqnarray*}\nonumber
g_{11}=\epsilon_1\bigg(\frac{f_1}{\lambda_2-\lambda_1}\bigg)^2,\hspace{1cm}g_{22}=\epsilon_2\bigg(\frac{f_2}{\lambda_2-\lambda_1}\bigg)^2,
\end{eqnarray*}
where $\epsilon_i^2=1$, $f_i$ $1\leq i\leq2$ are an arbitrary real functions of the
variable $u_i$.

Therefore, using the coordinate change $d\widetilde{u}_i=f_idu_i$, we have that $X$ is $\epsilon$-isothermic.
\end{proof}

From now on we will consider surfaces parametrized by lines of curvature with $\epsilon$-isothermic parameters whose first fundamental
form is (\ref{1forma}) and obtain all surfaces Dupin, with two distinct principal curvatures $-\lambda_1$ and $-\lambda_2$.\\

\begin{theorem} {Let $X:U\subseteq \mathbb{R}^{2}\rightarrow E^{3}$, be a $\epsilon$-isothermic Dupin surface with two distinct principal curvatures $-\lambda_1$ and $-\lambda_2$. Then\\
(i) $\lambda_1\lambda_2=0$,\\
or \\
(ii)
\begin{enumerate}
  \item If $b_1\neq 0$ and $(1+b_1) \neq 0$, the principal curvatures are given by
  \begin{eqnarray}
  \label{eqa}
  \lambda_2&=& -\frac{\epsilon_1c_1}{b_1}+A_1e^{\sqrt{b_1\epsilon_1}u_1}+A_2e^{-\sqrt{b_1\epsilon_1}u_1},\\
  \nonumber
  \lambda_1&=&-\frac{\epsilon_1c_1}{1+b_1}+A_3e^{\sqrt{-\epsilon_2(1+b_1)} u_2}+A_4e^{-\sqrt{-\epsilon_2(1+b_1)}u_2}.
  \end{eqnarray}
  where
  \begin{equation}
  \label{eqb}
  c_1^2-4b_1 (1+b_1) (b_1 A_1 A_2-(1+b_1) A_3 A_4)=0.
  \end{equation}
  \item If $b_1=0$, the principal curvatures are given by
  \begin{eqnarray}
  \label{eqc}
  \lambda_2&=& \frac{c_1}{2} u_1^2 + A_5 u_1 +A_6,\\
  \nonumber
  \lambda_1&=&- \epsilon_1c_1 + A_7 \cos{\sqrt{\epsilon_2} u_2} + A_8 \sin{\sqrt{\epsilon_2} u_2}.
  \end{eqnarray}
  where
  \begin{equation}
  \label{eqd}
  \epsilon_1A_5^2-c_1^2 -2\epsilon_1c_1 A_6+A_7^2+A_8^2=0.
  \end{equation}
  \item If $b_2\neq 0$ and $(1+b_2) \neq 0$, the principal curvatures are given by
  \begin{eqnarray}
  \label{eqe}
  \lambda_2&=&-\frac{\epsilon_2c_2}{1+b_2}+B_1e^{\sqrt{-\epsilon_1(1+b_2)} u_1}+B_2e^{-\sqrt{-\epsilon_1(1+b_2)}u_1},\\
  \nonumber
  \lambda_1&=& -\frac{\epsilon_2c_2}{b_2}+B_3e^{\sqrt{b_2\epsilon_2}u_2}+B_4e^{-\sqrt{b_2\epsilon_2}u_2}.
 \end{eqnarray}
  where
  \begin{equation}
  \label{eqf}
  c_2^2+4b_2 (1+b_2) ((1+b_2) B_1 B_2-b_2 B_3 B_4)=0.
  \end{equation}
  \item If $b_2=0$, the principal curvatures are given by
  \begin{eqnarray}
  \label{eqg}
  \lambda_2&=&- \epsilon_2c_2 + B_5 \cos{\sqrt{\epsilon_1} u_1} + B_6 \sin{\sqrt{\epsilon_1} u_1},\\
  \nonumber
 \lambda_1&=& \frac{c_2}{2} u_2^2 + B_7 u_2 +B_8.
   \end{eqnarray}
 where
  \begin{equation}
  \label{eqh}
  \epsilon_2B_7^2-c_2^2 -2\epsilon_2c_2 B_8+B_5^2+B_6^2=0.
  \end{equation}
\end{enumerate}
Conversely, if $\lambda_i:U\subseteq R^{2}\rightarrow R$, $1\leq i\leq2$, be real functions, distinct at each point, satisfy
the condition $(i)$ or $(ii)$, then there is a $\epsilon$-isothermic Dupin surface $X:U\subseteq R^{2}\rightarrow E^{3}$, whose principal curvatures are the
functions $-\lambda_i$.}
\end{theorem}

\begin{proof} From Theorem 2, the first fundamental form of
$X$ is  given by (\ref{1forma}) and by Gauss equation (\ref{gaussequation}) with $\varphi=\frac{-1}{2}\ln(\lambda_2-\lambda_1)^2$, using (\ref{codazi}) we have
\begin{eqnarray}\label{gauss2}
\frac{\lambda_1\lambda_2}{\big(\lambda_2-\lambda_1\big)^2}+\epsilon_2\Gamma^1_{12,2}+\epsilon_1\Gamma^2_{12,1}=0.
\end{eqnarray}
If $\lambda_1$ and $\lambda_2$ are constant, then using
(\ref{codazi}), we have $\Gamma^1_{12}=0$ and $\Gamma^2_{12}=0$.
Thus we obtain $\lambda_1\lambda_2=0$.

If $\lambda_1=h_2(u_2)$ and $\lambda_2=h_1(u_1)$, with
$h'_1(u_1)\neq 0$, using (\ref{codazi}), we obtain
\begin{eqnarray*}\nonumber
&&\Gamma^2_{12,1}=\bigg(\frac{h'_1}{h_2-h_1}\bigg)_{,1}=\frac{h''_1}{h_2-h_1}+\bigg(\frac{h'_1}{h_2-h_1}\bigg)^2,\nonumber\\
&&\Gamma^1_{12,2}=\bigg(\frac{h'_2}{h_1-h_2}\bigg)_{,2}=\frac{h''_2}{h_1-h_2}+\bigg(\frac{h'_2}{h_1-h_2}\bigg)^2.\nonumber
\end{eqnarray*}
Therefore,(\ref{gauss2}) can be rewritten as,
\begin{eqnarray}
\label{gauss21}
h_1h_2+\epsilon_2h''_2(h_1-h_2)+\epsilon_1h''_1(h_2-h_1)+\epsilon_1\big(h'_1\big)^2+\epsilon_2\big(h'_2\big)^2=0.
\end{eqnarray}
Differentiating (\ref{gauss21}) with respect to $u_1$ and using that
$h'_1\neq 0$, we have
\begin{eqnarray}\label{eq2}
\frac{\epsilon_1h'''_1}{h'_1}=\frac{1}{h_1-h_2}\bigg[h_2+\epsilon_2h''_2+\epsilon_1h''_1\bigg].
\end{eqnarray}
Differentiating (\ref{eq2}) with respect to $u_1$, we get
\begin{eqnarray*}\nonumber
\bigg(\epsilon_1\frac{h'''_1}{h'_1}\bigg)_{,1}=0.
\end{eqnarray*}
Therefore
\begin{eqnarray}
\label{eqh1}
h''_1-\epsilon_1b_1h_1=c_1,
\end{eqnarray}
where $b_1$ and $c_1$ are constants.\\
Substituting (\ref{eqh1}) in (\ref{eq2}), we have
\begin{eqnarray}
\label{eqh2}
h''_2+\epsilon_2(1+b_1)h_2=-\epsilon_1\epsilon_2c_1.
\end{eqnarray}
\begin{enumerate}
  \item If $b_1\neq 0$ and $(1+b_1) \neq 0$, the solutions of (\ref{eqh1}) and (\ref{eqh2}) are given by (\ref{eqa}).
  Using (\ref{eqa}) in (\ref{gauss21}) we get (\ref{eqb}).
  \item If $b_1=0$, the solutions of (\ref{eqh1}) and (\ref{eqh2}) are given by (\ref{eqc}).
  Using (\ref{eqc}) in (\ref{gauss21}) we get (\ref{eqd}).\\
If $\lambda_1=h_2(u_2)$ and $\lambda_2=h_1(u_1)$, with
$h'_2(u_2)\neq 0$, with similarly calculus, differentiating (\ref{gauss21}) with respect $u_2$, we obtain
\begin{eqnarray}
\label{eqA}
h''_2-\epsilon_2b_2h_2=c_2,\hspace{1cm}
h''_1+\epsilon_1(1+b_2)h_1=-\epsilon_2\epsilon_1c_2.
\end{eqnarray}
  \item If $b_2\neq 0$ and $(1+b_2) \neq 0$, the solutions of (\ref{eqA}) are given by (\ref{eqe}).\\
  Using (\ref{eqe}) in (\ref{gauss21}) we get (\ref{eqf}).
  \item If $b_2=0$, the solutions of (\ref{eqA}) are given by (\ref{eqg}). Using (\ref{eqg}) in (\ref{gauss21}) we get (\ref{eqh}).
\end{enumerate}
\end{proof}

\begin{theorem} {Let $X$ be a Dupin surface as in Theorem 3. If $-\lambda_1$ and $-\lambda_2$ are constant then $X$ is a cylinder.}
\end{theorem}

\begin{proof} Let $X$ be a Dupin surface as in Theorem 3 with constant principal curvature $-\lambda_1$ and
$-\lambda_2$. From Theorem 3, $\lambda_1\lambda_2=0$. If $\lambda_2=0$ and $\lambda_1=c\neq0$, then from Theorem 2, the
first fundamental form of $X$ is
$I=\frac{1}{c^2}\bigg(\epsilon_1du_1^2+\epsilon_2du_2^2\bigg)$,
$\epsilon_i^2=1$, $1\leq i\leq2$ and $X$ satisfy
\begin{eqnarray}\nonumber
&&X,_{12}=(0,0,0),\nonumber\\
&&X,_{11}=\frac{-\epsilon_1N}{c},\nonumber\\
&&X,_{22}=(0,0,0),\label{xcilindro}\\
&&N,_1=cX,_1,\nonumber\\
&&N,_2=(0,0,0).\nonumber
\end{eqnarray}
Using the last two equations of (\ref{xcilindro}), we have
\begin{eqnarray*}\nonumber
N=cX+H_2\,\,\mbox{and}\,N=H_1,\, \mbox{where}\, H_i=H_i(u_i), \,i=1,2, \mbox{are vector valued functions}.\nonumber
\end{eqnarray*}
Therefore
\begin{eqnarray}\label{cilindro}
X=G_1-G_2,
\end{eqnarray}
where $G_i=\frac{H_i}{c}$,\,$i=1,2$ are vector valued functions in $E^3$.

Differentiating (\ref{cilindro}), we have $X,_{11}=G''_1$ and
$X,_{22}=-G''_2$. \\
Thus, follows from the second and thirty
equation from (\ref{xcilindro}) that $G''_1=-\epsilon_1G_1$ and
$G_2''=0$.

Finally, giving initial conditions $X,_i(0,0)=\frac{e_i}{c},$ $i=1,2,\,$ $X(0,0)=(0,0,0)$, $N(0,0)=e_3$, we get
$G'_1(0)=\frac{e_1}{c}$, $G'_2(0)=-\frac{e_2}{c}$ and
$G_i(0)=\frac{e_3}{c}$.\\
So, if $\epsilon_1=1$, we have
\begin{eqnarray*}\nonumber
G_2=-\frac{e_2}{c}u_2+\frac{e_3}{c}\hspace{1cm}\mbox{and}\hspace{1cm}G_1=\frac{e_1}{c}\sin u_1 +\frac{e_3}{c}\cos u_1 ,\nonumber
\end{eqnarray*}
where $e_1=(1,0,0)$, $e_2=(0,1,0)$ and $e_3=(0,0,1)$.\\
Therefore, using (\ref{cilindro}), we get
\begin{eqnarray*}\nonumber
X=\frac{1}{c}\bigg(\sin u_1 ,u_2,\cos u_2 -1\bigg),\nonumber
\end{eqnarray*}
in this case, $X$ is a cylinder.\\
If $\epsilon_1=-1$, we have
\begin{eqnarray*}\nonumber
G_2=-\frac{e_2}{c}u_2+\frac{e_3}{c}\hspace{1cm}\mbox{and}\hspace{1cm}G_1=\frac{e_1}{c}\sinh u_1 +\frac{e_3}{c}\cosh u_1 ,\nonumber
\end{eqnarray*}
where $e_1=(1,0,0)$, $e_2=(0,1,0)$ and $e_3=(0,0,1)$.\\
Therefore, using (\ref{cilindro}), we get
\begin{eqnarray*}\nonumber
X=\frac{1}{c}\bigg(\sinh u_1 ,u_2,\cosh u_2 -1\bigg).\nonumber
\end{eqnarray*}
Thus, $X$ is a cylinder. The proof is complete.
\end{proof}

\begin{theorem} {Let $X$ be a Dupin surface as in Theorem 3. If $-\lambda_1$ or $-\lambda_2$ are not constant,
then $X$  is given by
\begin{eqnarray}\label{xgeral}
X=\frac{G_2-G_1}{h_1-h_2},
\end{eqnarray}
where the vector valued functions $G_i(x_i)$, $1\leq i\leq2$, satisfy
\begin{eqnarray}\nonumber
&&G''_i-\epsilon_ib_iG_i=v_i,\nonumber\\
&&v_1=\frac{\epsilon_1}{h_2(0)-h_1(0)}\bigg(-\epsilon_1h'_1(0) ,
-\epsilon_2h'_2(0) ,
-h_2(0)-b_1(h_2(0)-h_1(0))\bigg),\label{xcondicao}\\
&&v_2=-\epsilon_1\epsilon_2v_1,\nonumber\\
&&G_i(0)=e_3,\hspace{1cm}G_i'(0)=e_i,\nonumber\\
\end{eqnarray}
where $h_1=\lambda_2$, $h_2=\lambda_1$ and constant $b_i$, $1\leq i\leq2$, are given by Theorem 3.

Conversely, if the vector valued functions $G_i$ are given by (\ref{xcondicao}) with $h_i$, $1\leq i\leq2$, given by Theorem 3. Then (\ref{xgeral}) is a $\epsilon$-isothermic Dupin surface whose principal curvature are the functions $-\lambda_i$ where $\lambda_1=h_2(u_2)$ and $\lambda_2=h_1(u_1)$.}
\end{theorem}

\begin{proof} Let $X$ be a Dupin surface as in Theorem 3 with principal curvatures $-\lambda_i$ where
$\lambda_1=h_2(u_2)$ and $\lambda_2=h_1(u_1)$, $h'_1\neq0$. From Theorem 2, the first fundamental form of $X$ is given by
$I=\frac{1}{(h_1-h_2)^2}\bigg(\epsilon_1du_1^2+\epsilon_2du_2^2\bigg)$, $\epsilon_i^2=1$, $1\leq i\leq2$ and $X$ satisfy
\begin{eqnarray}\nonumber
&&X,_{12}-\Gamma^2_{12}X,_2-\Gamma^1_{12}X,_1=0,\nonumber\\
&&X,_{11}-\Gamma^2_{12}X,_1+\epsilon_1\epsilon_2\Gamma^1_{12}X,_2+\frac{\epsilon_1h_2N}{(h_2-h_1)^2}=0,\nonumber\\
&&X,_{22}-\Gamma^1_{12}X,_2+\epsilon_1\epsilon_2\Gamma^2_{12}X,_1+\frac{\epsilon_2h_1N}{(h_2-h_1)^2}=0 ,\label{eqxx}\\
&&N,_1=h_2X,_1,\nonumber\\
&&N,_2=h_1X,_2.\nonumber
\end{eqnarray}
Using (\ref{codazi}), we obtain
\begin{eqnarray}
&&\Gamma^2_{12}=\frac{h'_1}{h_2-h_1},\hspace{1cm}\Gamma^1_{12}=\frac{h'_2}{h_1-h_2},\label{eq3}\\
&&\Gamma^2_{12,1}-\big(\Gamma^2_{12}\big)^2=\frac{h''_1}{h_2-h_1},\hspace{1cm}\Gamma^1_{12,2}-\big(\Gamma^1_{12}\big)^2=\frac{h''_2}{h_1-h_2}.\nonumber
\end{eqnarray}
The last two equations of (\ref{eqxx}), we get
\begin{eqnarray*}\nonumber
N=h_2X+G_2\hspace{1cm}\mbox{and}\hspace{1cm}N=h_1X+G_1.\nonumber
\end{eqnarray*}
Hence
\begin{eqnarray}\label{xen}
X=\frac{G_2-G_1}{h_1-h_2}\hspace{1cm}\mbox{and}\hspace{1cm}N=\frac{h_1G_2-h_2G_1}{h_1-h_2}.\label{xen}
\end{eqnarray}
Substituting (\ref{xen}) in the first equation of (\ref{eqxx}), we obtain identity.

Subtracting the second and third equation of (\ref{eqxx}), we have
\begin{eqnarray}\label{x11x22}
\epsilon_2X,_{11}-\epsilon_1X,_{22}=\frac{\epsilon_1\epsilon_2N}{h_1-h_2}+2\epsilon_2\Gamma^2_{12}X,_1-2\epsilon_1\Gamma^1_{12}X,_2.\label{x11x22}
\end{eqnarray}
On the other hand, differentiating $X$ given by (\ref{xen}) and using (\ref{eq3}), we obtain
\begin{eqnarray}\nonumber
&&X,_1=\Gamma^2_{12}X+\frac{G'_1}{h_2-h_1},\hspace{0,5cm}X,_2=\Gamma^1_{12}X+\frac{G'_2}{h_1-h_2},\nonumber\\
&&X,_{11}=\big(\Gamma^2_{12,1}-(\Gamma^2_{12})^2\big)X+2\Gamma^2_{12}X,_1+\frac{G''_1}{h_2-h_1},\label{eq4}\\
&&X,_{22}=\big(\Gamma^1_{12,2}-(\Gamma^1_{12})^2\big)X+2\Gamma^1_{12}X,_2+\frac{G''_2}{h_1-h_2}.\nonumber
\end{eqnarray}
So, using (\ref{x11x22}) and (\ref{eq3}), we get
\begin{eqnarray*}\nonumber
-N=\big(\epsilon_1h''_1+\epsilon_2h''_2\big)X+\epsilon_1G''_1+\epsilon_2G''_2.
\end{eqnarray*}
Substituting $X$ and $N$ given by (\ref{xen}) and using the Theorem 2.2, we obtain
\begin{eqnarray*}\nonumber
\epsilon_2\big(G''_2+\epsilon_2(1+b_1)G_2\big)=-\epsilon_1\big(G''_1-\epsilon_1b_1G_1\big).\nonumber
\end{eqnarray*}
Therefore, $X$ is given by (\ref{xgeral}) and the vector valued functions $G_i(x_i)$, $1\leq i\leq2$ satisfy (\ref{xcondicao}).

Now, substituting (\ref{eq4}) and (\ref{xen}), in the second equation of (\ref{eqxx}), using (\ref{eq3}), the Gauss equation, Theorem 3 and (\ref{xgeral}), we obtain
\begin{equation}
\label{eq5}
v_1h_2+\epsilon_1(1+b_1)h_2G_2+c_1G_2+\epsilon_1\epsilon_2h'_2G'_2-v_1h_1-\epsilon_1b_1h_1G_1-c_1G_1+h'_1G'_1=0.
\end{equation}
Note that
$v_1h_2+\epsilon_1(1+b_1)h_2G_2+c_1G_2+\epsilon_1\epsilon_2h'_2G'_2-v_1h_1-\epsilon_1b_1h_1G_1-c_1G_1+h'_1G'_1$,
is constant vectors. In fact, it is sufficient differentiate these
expression with respect to $u_1$ and $u_2$, using (\ref{eqh1}),
(\ref{eqh2}) and first equation of (\ref{xcondicao}).

Finally, given initial conditions
$X,_1(0,0)=\frac{-e_1}{h_1(0)-h_2(0)}$,
$X,_2(0,0)=\frac{e_2}{h_1(0)-h_2(0)}$, $N(0,0)=e_3$ and
$X(0,0)=(0,0,0)$, from (\ref{eq4}), (\ref{xen}) and (\ref{eq5}),
we have $G'_i(0)=e_i$, $G_i=e_3$, and $v_i$ is given on second
equation of (\ref{xcondicao}), $1\leq i\leq2$.

\end{proof}

\section{ Examples of Dupin surfaces and solutions of the pseudo-Calapso equation}

In this section, using the Theorems 3 and 5, we give examples of $\epsilon$-isothermic
Dupin surfaces in $E^3$ with two distinct principal curvatures and solutions for the pseudo-Calapso
equation.\\

\begin{corollary} {Let $X:U\subseteq \mathbb{R}^{2}\rightarrow E^{3}$, be a $\epsilon$-isothermic Dupin surface with two distinct principal curvatures $-\lambda_1$ and $-\lambda_2$. Then the functions $\,\omega=\displaystyle\frac{\epsilon_1\sqrt 2(\lambda_2+\lambda_1)}{2(\lambda_2-\lambda_1)},$ $\Omega=\displaystyle\frac{\epsilon_1\sqrt 2}{2}\,$, are solutions of the pseudo-Calapso equation where $\lambda_1$ and $\lambda_2$ are given in the Theorem 3.}
\end{corollary}
\begin{proof} The result it follows from Theorems 1 and 3.
\end{proof}

\begin{proposition} {If $X$ is a cylinder, then the solutions to the pseudo-Calapso equation are constant.}
\end{proposition}

\begin{proof} Let $X$ a cylinder, from Theorem 3, the principal curvatures are given by
$\lambda_1=c\neq0$, $\lambda_2=0$ and using the Theorem 2, the first fundamental form of $X$ is given by
$I=\frac{1}{c^2}\bigg(\epsilon_1du_1^2+\epsilon_2du_2^2\bigg)$, $\epsilon_i^2=1$, $1\leq i\leq2$.\\
So, from Theorem 1, the solutions of the pseudo-Calapso equation are given by
$\omega=\epsilon_1\sqrt{2}He^\varphi$ and $\Omega=\epsilon_1\sqrt{2}H'e^\varphi$, $\epsilon_1^2=1$.\\
Since, $H=-H'=\frac{c}{2}$ and $e^\varphi=\frac{1}{c}$, we get that the solutions are given by $\Omega=-\omega=-\frac{\epsilon_1\sqrt{2}}{2}$.
The proof is complete.
\end{proof}

\begin{example} {\em Consider $b_1=0$, then from Theorem 3, we get
\begin{eqnarray}\label{exemplo1}
\lambda_1=a_{21}f(u_2)+a_{22}g(u_2)-\epsilon_1c_1,\hspace{1cm}\lambda_2=\frac{c_1}{2}u^2_1+a_{11}u_1+a_{12},
\end{eqnarray}
where
\begin{eqnarray}\label{f}
f(u_2)=\left\{\begin{array}{ll}
&\sinh(u_2)\hspace{0,1cm} if \hspace{0,1cm}\epsilon_2=-1\\
&\sin(u_2),\hspace{0,1cm} if \hspace{0,1cm}\epsilon_2=1,
\end{array}
\right.
\end{eqnarray}
\begin{eqnarray}\label{g}
g(u_2)=\left\{\begin{array}{ll}
&\cosh(u_2)\hspace{0,1cm} if \hspace{0,1cm}\epsilon_2=-1\\
&\cos(u_2),\hspace{0,1cm} if \hspace{0,1cm}\epsilon_2=1,
\end{array}
\right.
\end{eqnarray}
where the constants $c_1$, $a_{11}$, $a_{12}$, $a_{21}$ and
$a_{22}$, satisfy
\begin{eqnarray*}\nonumber
a^2_{22}+\epsilon_2a^2_{21}+\epsilon_1a^2_{11}-2\epsilon_1c_1a_{12}-c_1^2=0,\hspace{1cm}a_{22}-a_{12}-\epsilon_1c_1\neq0.
\end{eqnarray*}
Using Theorem 5 and (\ref{xen}), we get a $\epsilon$-isothermic Dupin surface
$X=\frac{G_2-G_1}{\lambda_2-\lambda_1}$ with
\begin{eqnarray*}
\nonumber
N&=&\frac{\lambda_2G_2-\lambda_1G_1}{\lambda_2-\lambda_1},\\
\nonumber
v_1&=&\frac{1}{a_{22}-a_{12}-\epsilon_1c_1}\bigg(a_{11} ,
\epsilon_2\epsilon_1a_{21} , \epsilon_1a_{22}-c_1\bigg),\\
\nonumber
G_1&=&\frac{u_1^2}{2}v_1+u_1e_1+e_3,\nonumber\\
G_2&=&\epsilon_1\big(g(u_2)-1\big)v_1+f(u_2)e_2+g(u_2)e_3,\nonumber
\end{eqnarray*}
where $e_1=(1,0,0)$, $e_2=(0,1,0)$ and $e_3=(0,0,1)$.}
\end{example}

\begin{proposition} {Let $X$ Dupin surface as in Example 1. Then the functions
\begin{eqnarray}\label{calapso1}
&&\omega=\frac{\epsilon_1\sqrt{2}}{2}\frac{c_1u^2_1+2a_{11}u_1+2a_{12}+2a_{21}f(u_2)+2a_{22}g(u_2)+2\epsilon_1c_1}{c_1u^2_1+2a_{11}u_1+2a_{12}-2a_{21}f(u_2)+
2a_{22}g(u_2)-2\epsilon_1c_1},\\
&&\Omega=\frac{\epsilon_1\sqrt{2}}{2},\hspace{0,5cm}\epsilon_1^2=1\nonumber
\end{eqnarray}
are solutions for the pseudo-Calapso equation, where $f$ and $g$ are given
by (\ref{f}) and (\ref{g}) and the constants $c_1$, $a_{11}$,
$a_{12}$, $a_{21}$ and $a_{22}$, satisfy
\begin{eqnarray*}\nonumber
a^2_{22}+\epsilon_2a^2_{21}+\epsilon_1a^2_{11}-2\epsilon_1c_1a_{12}-c_1^2=0,\hspace{1cm}a_{22}-a_{12}-\epsilon_1c_1\neq0
\end{eqnarray*}}
\end{proposition}

\begin{proof} Let $X$ given by Example 1, where the principal curvature $-\lambda_1$ and $-\lambda_2$, are give by
(\ref{exemplo1}) and using the Theorem 2, the first fundamental form of $X$ is given by $I=\frac{1}{(\lambda_2-\lambda_1)^2}\bigg(\epsilon_1du_1^2+\epsilon_2du_2^2\bigg)$,
$\epsilon_i^2=1$, $1\leq i\leq2$.\\
So, from Theorem 1, the solutions of the pseudo-Calapso equation are given by
$\omega=\epsilon_1\sqrt{2}He^\varphi$ and $\Omega=\epsilon_1\sqrt{2}H'e^\varphi$. Since,
$H=\frac{\lambda_2+\lambda_1}{2}$, $H'=\frac{\lambda_2-\lambda_1}{2}$ and
$e^\varphi=\frac{1}{\lambda_2-\lambda_1}$, we get that the solutions are given by (\ref{calapso1}). The proof is complete.

\end{proof}

\begin{example}{\em Consider $-1<b_1<0$, then from Theorem 3, we get
\begin{eqnarray}\label{exemplo3}
&&\lambda_1=a_{21}f(u_2)+a_{22}g(u_2)-\frac{\epsilon_1c_1}{b_1+1},\\
&&\lambda_2=a_{11}\widetilde{f}(u_1)+a_{12}\widetilde{g}(u_1)-\frac{\epsilon_1c_1}{b_1}.\nonumber
\end{eqnarray}
where
\begin{eqnarray}\label{f1}
f(u_2)=\left\{\begin{array}{ll}
&\sinh(\sqrt{1+b_1}u_2)\hspace{0,1cm} if \hspace{0,1cm}\epsilon_2=-1\\
&\sin(\sqrt{1+b_1}u_2),\hspace{0,1cm} if
\hspace{0,1cm}\epsilon_2=1,
\end{array}
\right.
\end{eqnarray}
\begin{eqnarray}\label{g1}
g(u_2)=\left\{\begin{array}{ll}
&\cosh(\sqrt{1+b_1}u_2)\hspace{0,1cm} if \hspace{0,1cm}\epsilon_2=-1\\
&\cos(\sqrt{1+b_1}u_2),\hspace{0,1cm} if
\hspace{0,1cm}\epsilon_2=1,
\end{array}
\right.
\end{eqnarray}
\begin{eqnarray}\label{f1tiu}
\widetilde{f}(u_1)=\left\{\begin{array}{ll}
&\sinh(\sqrt{-b_1}u_1)\hspace{0,1cm} if \hspace{0,1cm}\epsilon_1=-1\\
&\sin(\sqrt{-b_1}u_2),\hspace{0,1cm} if
\hspace{0,1cm}\epsilon_1=1,
\end{array}
\right.
\end{eqnarray}
\begin{eqnarray}\label{g1tiu}
\widetilde{g}(u_1)=\left\{\begin{array}{ll}
&\cosh(\sqrt{-b_1}u_2)\hspace{0,1cm} if \hspace{0,1cm}\epsilon_1=-1\\
&\cos(\sqrt{-b_1}u_2),\hspace{0,1cm} if
\hspace{0,1cm}\epsilon_1=1,
\end{array}
\right.
\end{eqnarray}
where $\epsilon_i^2=1$, $\epsilon=\epsilon_1\epsilon_2$ and the
constants $c_1$, $a_{11}$, $a_{12}$, $a_{21}$ and $a_{22}$,
satisfy
\begin{eqnarray}
&&-b_1\big(\epsilon_1a_{11}^2+a_{12}^2\big)+\big(\epsilon_2a_{21}^2+a_{22}^2\big)\big(1+b_1\big)+\frac{c_1^2}{b_1( b_1+1)}=0,\label{ex3}\\
&&m_1=\frac{\big(a_{12}-a_{22}\big)\big(b^2_1+b_1\big)-\epsilon_1c_1}{b_1^2+b_1}\neq0.\nonumber
\end{eqnarray}
Using Theorem 5 and (\ref{xen}), we get a $\epsilon$-isothermic Dupin surface
$X=\frac{G_2-G_1}{\lambda_2-\lambda_1}$ with
\begin{eqnarray*}
N&=&\frac{\lambda_2G_2-\lambda_1G_1}{\lambda_2-\lambda_1},\\
v_1&=&\frac{1}{m_1}\bigg(-a_{11}\sqrt{-b_1}, \epsilon a_{21}\sqrt{b_1+1} ,
\epsilon_1\bigg(-b_1m_1+\frac{a_{22}(1+b_1)-\epsilon_1c_1}{1+b_1}\bigg)\bigg),\\
G_1&=&\widetilde{g}(u_1)\bigg(\frac{\epsilon_1v_1}{b_1}+e_3\bigg)+\frac{\widetilde{f}(u_1)}{\sqrt{-b_1}}e_1-\frac{\epsilon_1v_1}{b_1},\\
G_2&=&g(u_2)\bigg(\frac{\epsilon_1v_1}{b_1+1}+e_3\bigg)+\frac{f(u_2)}{\sqrt{b_1+1}}e_2-\frac{\epsilon_1v_1}{b_1+1}.
\end{eqnarray*}}
\end{example}

\begin{proposition} {Let $X$ Dupin surface as in Example 2. Then the functions
%\begin{eqnarray}\label{calapso31}
%&&\omega=\frac{\sqrt{2}}{2}\frac{(b_1^2+b_1)(a_{11}\widetilde{f}(u_1)+a_{12}\widetilde{g}(u_1)+a_{21}f(u_2)+a_{22}g(u_2))-\epsilon_1c_1(2b_1+1)}{(b_1^2+b_1)(a_{11}\widetilde{f}(u_1)+a_{12}\widetilde{g}(u_1)-a_{21}f(u_2)-a_{22}g(u_2))-\epsilon_1c_1}
%&&\Omega=\frac{\sqrt{2}}{2},\nonumber
%\end{eqnarray}
\begin{eqnarray}
\omega=&&\frac{\epsilon_1\sqrt{2}}{2}\frac{(b_1^2+b_1)(a_{11}\widetilde{f}(u_1)+a_{12}\widetilde{g}(u_1)+a_{21}f(u_2)+a_{22}g(u_2))-
\epsilon_1c_1(2b_1+1)}{(b_1^2+b_1)(a_{11}\widetilde{f}(u_1)+a_{12}\widetilde{g}(u_1)-a_{21}f(u_2)-a_{22}g(u_2))-\epsilon_1c_1},\nonumber\\
\,\,\,\,\,\,\,\,\,\,&&\Omega=\frac{\epsilon_1\sqrt{2}}{2},\label{calapso3}
\end{eqnarray}
are solutions for the pseudo-Calapso equation,  where $f$, $g$, $\widetilde{f}$ and $\widetilde{g}$ are given by (\ref{f1}),
(\ref{g1}), (\ref{f1tiu}) and (\ref{g1tiu}) and the constants $c_1$, $a_{11}$, $a_{12}$, $a_{21}$ and $a_{22}$, satisfy
(\ref{ex3}).}
\end{proposition}

\begin{proof} Let $X$ given by Example 2, where the principal curvature $-\lambda_1$ and $-\lambda_2$, are given by
(\ref{exemplo3}) and using the Theorem 2, the first fundamental form of  $X$ is given by
$I=\frac{1}{(\lambda_2-\lambda_1)^2}\bigg(\epsilon_1du_1^2+\epsilon_2du_2^2\bigg)$,
$\epsilon_i^2=1$, $1\leq i\leq2$.\\
So, from Theorem 1, the solutions of the pseudo-Calapso equation are given by
$\omega=\epsilon_1\sqrt{2}He^\varphi$ and $\Omega=\epsilon_1\sqrt{2}H'e^\varphi$.\\
Since, $H=\frac{\lambda_2+\lambda_1}{2}$, $H'=\frac{\lambda_2-\lambda_1}{2}$ and
$e^\varphi=\frac{1}{\lambda_2-\lambda_1}$, it follows that (\ref{calapso3}) are solutions for the pseudo-Calapso equation.

\end{proof}

\begin{example}{\em Consider $b_1>0$, then from Theorem 3, we get
\begin{eqnarray}\label{exemplo4}
&&\lambda_1=a_{21}f(u_2)+a_{22}g(u_2)-\frac{\epsilon_1c_1}{b_1+1},\\
&&\lambda_2=a_{11}\widehat{f}(u_1)+a_{12}\widehat{g}(u_1)-\frac{\epsilon_1c_1}{b_1}.\nonumber
\end{eqnarray}
where $f$ and $g$ are given by (\ref{f1}), (\ref{g1}) and
\begin{eqnarray}\label{f1chapel}
\widehat{f}(u_1)=\left\{\begin{array}{ll}
&\sinh(\sqrt{b_1}u_1)\hspace{0,1cm} if \hspace{0,1cm}\epsilon_1=-1\\
&\sin(\sqrt{b_1}u_2),\hspace{0,1cm} if \hspace{0,1cm}\epsilon_1=1,
\end{array}
\right.
\end{eqnarray}
\begin{eqnarray}\label{g1chapel}
\widehat{g}(u_1)=\left\{\begin{array}{ll}
&\cosh(\sqrt{b_1}u_2)\hspace{0,1cm} if \hspace{0,1cm}\epsilon_1=-1\\
&\cos(\sqrt{b_1}u_2),\hspace{0,1cm} if \hspace{0,1cm}\epsilon_1=1,
\end{array}
\right.
\end{eqnarray}
where $\epsilon_i^2=1$, $\epsilon=\epsilon_1\epsilon_2$ and the
constants $c_1$, $a_{11}$, $a_{12}$, $a_{21}$ and $a_{22}$,
satisfy
\begin{eqnarray}
&&b_1\big(\epsilon_1a_{11}^2-a_{12}^2\big)+\big(\epsilon_2a_{21}^2+a_{22}^2\big)\big(1+b_1\big)+\frac{c_1^2}{b_1( b_1+1)}=0,\label{ex4}\\
&&m_1=\frac{\big(a_{22}-a_{12}\big)\big(b^2_1+b_1\big)-\epsilon_1c_1}{b_1^2+b_1}\neq0.\nonumber
\end{eqnarray}
Using Theorem 5 and (\ref{xen}), we get $\epsilon$-isothermic Dupin surface
$X=\frac{G_2-G_1}{\lambda_2-\lambda_1}$ with
\begin{eqnarray*}
N&=&\frac{\lambda_2G_2-\lambda_1G_1}{\lambda_2-\lambda_1}\\
\nonumber
v_1&=&\frac{1}{m_1}\bigg(-a_{11}\sqrt{b_1} , -\epsilon a_{21}\sqrt{b_1+1} , \epsilon_1\bigg(-b_1m_1+\frac{a_{22}(-1-b_1)+\epsilon_1c_1}{1+b_1}\bigg)\bigg),\\
G_1&=&\widehat{g}(u_1)\bigg(\frac{\epsilon_1v_1}{b_1}+e_3\bigg)+\frac{\widehat{f}(u_1)}{\sqrt{b_1}}e_1-\frac{\epsilon_1v_1}{b_1},\nonumber\\
G_2&=&g(u_2)\bigg(\frac{\epsilon_1v_1}{b_1+1}+e_3\bigg)+\frac{f(u_2)}{\sqrt{b_1+1}}e_2-\frac{\epsilon_1v_1}{b_1+1}.\nonumber
\end{eqnarray*}}
\end{example}

\begin{proposition}{Let $X$ Dupin surface as in Example 3. Then the functions
%\begin{eqnarray}\label{calapso2}
%&&\omega=\frac{\sqrt{2}}{2}\frac{(b_1^2+b_1)\big(a_{11}\sinh(\sqrt{b_1}u_1)+a_{12}\cosh(\sqrt{b_1}u_1)+a_{21}\sin(\sqrt{b_1+1}u_2)+a_{22}\cos(\sqrt{b_1+1}u_2)\big)-c_1(2b_1+1)}{(b_1^2+b_1)\big(a_{11}\sinh(\sqrt{b_1}u_1)+a_{12}\cosh(\sqrt{b_1}u_1)-a_{21}\sin(\sqrt{b_1+1}u_2)-a_{22}\cos(\sqrt{b_1+1}u_2)\big)+c_1},\\
%&&\Omega=\frac{\sqrt{2}}{2},\nonumber
%\end{eqnarray}
\begin{eqnarray}
\omega=&&\frac{\epsilon_1\sqrt{2}}{2}\frac{(b_1^2+b_1)(a_{11}\widehat{f}(u_1)+a_{12}\widehat{g}(u_1)+a_{21}f(u_2)+a_{22}g(u_2))-
\epsilon_1c_1(2b_1+1)}{(b_1^2+b_1)(a_{11}\widehat{f}(u_1)+a_{12}\widehat{g}(u_1)-a_{21}f(u_2)-a_{22}g(u_2))-\epsilon_1c_1},\nonumber\\
&&\Omega=\frac{\epsilon_1\sqrt{2}}{2},\label{calapso4}
\end{eqnarray}
are solutions for the pseudo-Calapso equation, where $f$, $g$, $\widehat{f}$ and $\widehat{g}$ are given by
(\ref{f1}), (\ref{g1}), (\ref{f1chapel}) and (\ref{g1chapel}) and
the constants $c_1$, $a_{11}$, $a_{12}$, $a_{21}$ and $a_{22}$,
satisfy (\ref{ex4}).}
\end{proposition}
\begin{proof} Let $X$ given by Example 3, where the principal curvatures $-\lambda_1$ and $-\lambda_2$, are give by
(\ref{exemplo4}) and using the Theorem 2, the first fundamental form is given by
$I=\frac{1}{(\lambda_2-\lambda_1)^2}\bigg(\epsilon_1du_1^2+\epsilon_2du_2^2\bigg)$,
$\epsilon_i^2=1$, $1\leq i\leq2$.\\
So, from Theorem 1, the solutions of the pseudo-Calapso equation are
$\omega=\epsilon_1\sqrt{2}He^\varphi$ and $\Omega=\epsilon_1\sqrt{2}H'e^\varphi$. Since, $H=\frac{\lambda_2+\lambda_1}{2}$,
$H=\frac{\lambda_2-\lambda_1}{2}$ and
$e^\varphi=\frac{1}{\lambda_2-\lambda_1}$, we obtain that the solutions
are given by (\ref{calapso4}). The proof is complete.

\end{proof}

\begin{proposition}{If $f$ is a $\epsilon_2-$holomorphic function,  $\epsilon_2=\pm1$, then the function $\omega$
given by
\begin{eqnarray*}\nonumber
\omega=\frac{2\sqrt{2|\langle f',f'\rangle |}}{1+|f|^2}
\end{eqnarray*}
is a solution to the pseudo-Calapso equation.}
\end{proposition}

\begin{proof} Let $\epsilon_2=\pm1$ and $f$ is a $\epsilon_2-$holomorphic function and define the application
\begin{eqnarray*}\nonumber
X(z)=\bigg(\frac{2f}{1+\epsilon_3\langle f,f\rangle},\frac{\epsilon_3\langle f,f\rangle -1}{1+\epsilon_3\langle f,f\rangle}\bigg),\hspace{0,5cm}z\in
C_{\epsilon}
\end{eqnarray*}
where $\epsilon_3^2=1$.\\
This application is a parametrization of the sphere in $\mathbb{R}^3$ with
metric $du_1^2+\epsilon_2 du_2^2+\epsilon_3du_3^2$. In fact, it is easy to see that $\langle X,X\rangle=\epsilon_3$. Moreover, the first fundamental
form of $X$ is given by
\begin{eqnarray*}\nonumber
I=\frac{4|\langle f',f'\rangle|}{\big(1+\epsilon_3|f|^2\big)^2}\big[du_1^2+\epsilon_2 du_2^2\big].
\end{eqnarray*}
Using Theorem 1, we get the result.

\end{proof}

\end{document}